\documentclass[11pt]{amsart}

\usepackage{latexsym,rawfonts}
\usepackage{amsfonts,amssymb}
\usepackage{amsmath,amsthm}

\usepackage[plainpages=false]{hyperref}
\usepackage{graphicx}
\usepackage{color}
\usepackage[table]{xcolor}
\usepackage{longtable}

\newtheorem{theorem}{Theorem}[section]

\newtheorem{lemma}{Lemma}[section]

\theoremstyle{definition}

\numberwithin{equation}{section}

\DeclareMathOperator{\VOL}{vol}

\DeclareMathOperator{\DIAG}{diag}

\DeclareMathOperator{\SPAN}{span}
\newcommand{\R}{\mathbb{R}}

\newcommand{\K}{\mathcal{K}}
\newcommand{\dd}{\mathop{}\!\mathrm{d}}
\newcommand{\abs}[1]{\left\vert#1\right\vert}
\newcommand{\set}[1]{\left\{#1\right\}}
\newcommand{\norm}[1]{\left\Vert#1\right\Vert}

\newcommand{\pd}{\partial}
\newcommand{\delbar}{\overline{\nabla}}

\newcommand{\uS}{\mathbb{S}^{n-1}}

\newcommand{\MA}{Monge-Amp\`ere }

\newcommand{\CTq}{\widetilde{C}_q}
\newcommand{\VTq}{\widetilde{V}_q}

\begin{document}

\title{Existence of solutions to the generalized dual Minkowski problem}

\author{Mingyang Li}
\address{School of Mathematical Sciences, South China Normal University, Guangzhou 510631, P.R. China}
\email{1519514482@qq.com}

\author{YanNan Liu}
\address{School of Mathematics and Statistics, Beijing Technology and Business University, Beijing 100048, P.R. China}
\email{liuyn@th.btbu.edu.cn}

\author{Jian Lu}
\address{School of Mathematical Sciences, South China Normal University, Guangzhou 510631, P.R. China}
\email{lj-tshu04@163.com}

\thanks{The author Liu was partially supported by
  Natural Science Foundation of China (12071017 and 12141103).} 

\thanks{The author Lu was partially supported by
  Natural Science Foundation of China (12122106).}

\date{}

\begin{abstract}
Given a real number $q$ and a star body in the $n$-dimensional Euclidean space,
the generalized dual curvature measure of a convex body was introduced by
Lutwak-Yang-Zhang \cite{LYZ.Adv.329-2018.85}. The corresponding generalized dual
Minkowski problem is studied in this paper. By using variational methods, we
solve the generalized dual Minkowski problem for $q<0$, and the even generalized
dual Minkowski problem for $0\leq q\leq1$. We also obtain a sufficient condition
for the existence of solutions to the even generalized dual Minkowski problem
for $1<q<n$.
\end{abstract}

\keywords{
  Dual Minkowski problem,
  Monge-Amp\`ere equation,
  subspace mass inequality,
  integral estimate.
}

\makeatletter
\@namedef{subjclassname@2020}{\textup{2020} Mathematics Subject Classification}
\makeatother

\subjclass[2020]{52A20, 35J96.}

\maketitle
\vskip4ex

\section{Introduction}

Given a real number $q\in\R$ and
a star body $Q$ in the $n$-dimensional Euclidean space $\R^n$,
for any convex body $K\subset\R^n$ containing the origin in its interior,
its \emph{generalized $q$-th dual curvature measure} $\widetilde{C}_q(K,Q,\cdot)$
is defined as
\begin{equation*}
  \widetilde{C}_q(K,Q,\eta)
  = \frac{1}{n} \int_{\boldsymbol{\alpha}_K^*(\eta)} \rho_K^q(u) \rho_Q^{n-q}(u) \dd u,
\end{equation*}
where $\eta$ is any Borel subset of the unit sphere $\uS$,
$\boldsymbol{\alpha}_K^*$ is the reverse radial Gauss image,
and $\rho_K,\rho_Q$ are radial functions of $K,Q$ respectively.
This definition was introduced by Lutwak-Yang-Zhang \cite{LYZ.Adv.329-2018.85}.
The corresponding \emph{generalized dual Minkowski problem}
is to find necessary and sufficient conditions on
a finite Borel measure $\mu$ on $\uS$, such that
\begin{equation}\label{gdMP}
  \mu=\widetilde{C}_q(K,Q,\cdot) 
\end{equation}
holds for some convex body $K\subset\R^n$.

In the special case when the given measure $\mu$ has a density $\frac{1}{n}f$
with respect to the standard measure on $\uS$,
the generalized dual Minkowski problem \eqref{gdMP}
is equivalent to solving the following \MA type equation: 
\begin{equation*} 
  h \Vert \delbar h \Vert_Q^{q-n} \det(\nabla^2h+hI) = f \quad \text{ on } \ \uS,
\end{equation*}
where $h$ is the support function of some convex body $K$,
$\nabla$ is the covariant derivative with respect to an orthonormal frame on $\uS$,
$\delbar h(x)=\nabla h(x) +h(x)x$ is the point on $\pd K$
whose unit outer normal vector is $x\in\uS$,
$\norm{\cdot}_Q$ is the Minkowski functional given by
\begin{equation*}
  \norm{y}_Q =\inf\set{\lambda>0 : y\in\lambda Q},
  \quad \forall y\in\R^n,
\end{equation*}
$I$ is the unit matrix of order $n-1$, and
$f$ is a given nonnegative integrable function.

When $Q$ is the unit ball, $\widetilde{C}_q(K,Q,\cdot)$
is reduced to the $q$-th dual curvature measure $\widetilde{C}_q(K,\cdot)$,
and Eq. \eqref{gdMP} reduced to the dual Minkowski problem.
As we know, the dual Minkowski problem 
was first proposed and studied by Huang-Lutwak-Yang-Zhang in
their groundbreaking paper \cite{HLYZ.Acta.216-2016.325}.
It contains two important special cases.
One is the logarithmic Minkowski problem when $q=n$;
see e.g. 
\cite{BHZ.IMRNI.2016.1807,
  BLYZ.JAMS.26-2013.831,
  CFL.Adv.411-2022.108782,
  CLZ.TAMS.371-2019.2623,
  KM.MAMS.277-2022.1,
  Sta.Adv.180-2003.290,
  Zhu.Adv.262-2014.909}.
The other is the Alexandrov problem when $q=0$,
which is the prescribed Alexandrov integral curvature problem
\cite{Ale.CRDASUN.35-1942.131, HLYZ.JDG.110-2018.1}.
In recent years,
the dual Minkowski problem has attracted great attention from many researchers;
see e.g.
\cite{BHP.JDG.109-2018.411,
  BLY+.Adv.356-2019.106805,
  CL.Adv.333-2018.87,
  EH.AiAM.151-2023.25,
  HP.Adv.323-2018.114,
  HJ.JFA.277-2019.2209,
  JW.JDE.263-2017.3230,
  LSW.JEMSJ.22-2020.893,
  LL,
  Zha.CVPDE.56-2017.18,
  Zha.JDG.110-2018.543}.

When $Q$ is a general star body, the uniqueness of solutions to Eq. \eqref{gdMP}
has recently been proved when $q<0$;
see \cite[Theorem 8.3]{LYZ.Adv.329-2018.85} for a discrete measure $\mu$,
and \cite[Theorem 1.2]{WZ.CMS-2024} for a general measure $\mu$.
Note that when $q=n$,
$\widetilde{C}_q(K,Q,\cdot)$ is independent of $Q$ by definition, 
and Eq. \eqref{gdMP} is then reduced to the logarithmic Minkowski problem,
whose even case was completely solved in \cite{BLYZ.JAMS.26-2013.831}.

In this paper we are concerned with
the existence of solutions to the generalized dual Minkowski problem
\eqref{gdMP}.
Recall that a measure on $\uS$ is called to be \emph{even},
if it has the same value on antipodal measurable subsets of $\uS$.

When $1<q<n$, 
a sufficient condition for the existence of origin-symmetric solutions to Eq. \eqref{gdMP}
is obtained.

\begin{theorem}\label{thm1} 
Assume $1<q<n$, and $Q$ is an origin-symmetric star body in $\R^n$.
If $\mu$ is a finite even Borel measure on $\uS$ satisfying
the following $q$-th subspace mass inequality:
\begin{equation*} 
  \frac{\mu(\uS\cap\xi_i)}{\mu(\uS)}<
  \min\set{\frac{i}{q},1}
\end{equation*}
for any proper $i$-dimensional subspace $\xi_i\subset\R^n$ with
$i=1,\cdots,n-1$,
then there exists an origin-symmetric convex body $K$ in $\R^n$
such that $\widetilde{C}_q(K,Q,\cdot)=\mu$.
\end{theorem}

When $0\leq q\leq1$, Eq. \eqref{gdMP} for the even case is completely solved.

\begin{theorem}\label{thm2} 
Assume $0<q\leq1$, $Q$ is an origin-symmetric star body in $\R^n$,
and $\mu$ is a finite even Borel measure on $\uS$.
Then there exists an origin-symmetric convex body $K$ in $\R^n$
such that $\widetilde{C}_q(K,Q,\cdot)=\mu$
if and only if $\mu$
is not concentrated on any great sub-sphere of $\uS$.
\end{theorem}

\begin{theorem}\label{thm3} 
Assume $Q$ is an origin-symmetric star body in $\R^n$,
and $\mu$ is a finite even Borel measure on $\uS$.
Then there exists an origin-symmetric convex body $K$ in $\R^n$
such that $\widetilde{C}_0(K,Q,\cdot)=\mu$
if and only if $\mu$
is not concentrated on any great sub-sphere of $\uS$
and $\mu(\uS)$ is equal to the volume of $Q$.
\end{theorem}

When $q<0$, Eq. \eqref{gdMP} can be solved for the general case.

\begin{theorem}\label{thm4} 
Assume $q<0$, $Q$ is a star body in $\R^n$,
and $\mu$ is a finite Borel measure on $\uS$.
Then there exists a convex body $K$ in $\R^n$
containing the origin in its interior,
such that $\widetilde{C}_q(K,Q,\cdot)=\mu$
if and only if $\mu$
is not concentrated in any closed hemisphere of $\uS$.
\end{theorem}

We note that when $Q$ is the unit ball $B^n$, the above four theorems have been
obtained in previous literature.
Specifically, in the case $Q=B^n$,
Theorem \ref{thm1} was established in
\cite{Zha.JDG.110-2018.543, BLY+.Adv.356-2019.106805},
Theorem \ref{thm2} was proved in
\cite{HLYZ.Acta.216-2016.325},
Theorem \ref{thm3} was proved in
\cite{Ale.CRDASUN.35-1942.131, HLYZ.JDG.110-2018.1},
and Theorem \ref{thm4} was obtained in
\cite{Zha.CVPDE.56-2017.18}.
We also note that when $Q=B^n$, the $q$-th subspace mass inequality
in Theorem \ref{thm1} is necessary \cite{BHP.JDG.109-2018.411}.
When $Q$ is a general star body, these above theorems are new, as far as we know.

Our methods of proving Theorems \ref{thm1}---\ref{thm4} are 
the variational methods developed in several papers
\cite{BLYZ.JAMS.26-2013.831,
  HLYZ.Acta.216-2016.325,
  Zha.CVPDE.56-2017.18,
  HLYZ.JDG.110-2018.1,
  Zha.JDG.110-2018.543,
  BLY+.Adv.356-2019.106805}.
When proving Theorems \ref{thm1} and \ref{thm2},
a sharp estimate about dual quermassintegrals of any origin-symmetric convex body
is crucial.
By utilizing the maximum-volume ellipsoid of an origin-symmetric convex body,
it is equivalent to finding a sharp estimate about dual quermassintegrals of any
origin-centered ellipsoid.
In these mentioned papers, several different types of barrier bodies were constructed
to estimate dual quermassintegrals of ellipsoids, such as
a cross-polytope in \cite{HLYZ.Acta.216-2016.325},
the Cartesian product of an ellipsoid and a ball in \cite{Zha.JDG.110-2018.543}, and
the Cartesian product of an ellipsoid, a line segment, and a ball 
in \cite{BLY+.Adv.356-2019.106805}.
While in our paper,
by directly estimating the integral expression of dual quermassintegrals of
origin-centered ellipsoids,
we can obtain a bidirectional sharp estimate;
see Lemmas \ref{lemInt} and \ref{lem314}.
In fact, the key technique to prove Lemma \ref{lemInt} comes from
\cite[Lemma 4.1]{JLW.JFA.274-2018.826}
written by Jian, Wang, and the third author.

At the end of this introduction, we remark that
there are various other extensions of the dual Minkowski problem, such as
$L_p$ dual Minkowski problem
\cite{BF.JDE.266-2019.7980,
  CHZ.MA.373-2019.953,
  CCL.AP.14-2021.689,
  HZ.Adv.332-2018.57,
  JWW.CVPDE.60-2021.16,
  JWW.MA.386-2023.1201,
  LLL.IMRNI.-2022.9114,
  WZ.CMS-2024}, 
dual Orlicz-Minkowski problem
\cite{CLLX.JGA.32-2022.40,
  GHW+.CVPDE.58-2019.12,
  GHXY.CVPDE.59-2020.15,
  LL.TAMS.373-2020.5833,
  ZXY.JGA.28-2018.3829},
and Gaussian Minkowski problem
\cite{FLX.JGA.33-2023.39,
  FHX.JDE.363-2023.350,
  HXZ.Adv.385-2021.36,
  Liu.CVPDE.61-2022.23}.
See also
\cite{Che.ANS.23-2023.20220068,
  HLYZ.Adv.224-2010.2485,
  JL.Adv.344-2019.262,
  Jia.ANS.10-2010.297,
  JLL.ANS.21-2021.155,
  LS.ANS.23-2023.20220040,
  LW.JDE.254-2013.983,
  Lut.JDG.38-1993.131,
  XYZZ.ANS.23-2023.20220041}
for other Minkowski type problems.

This paper is organized as follows.
In Section \ref{sec2},
we give some basic knowledge about convex bodies and dual curvature measures.
In Section \ref{sec3},
a key integral estimate is proved which will be used to obtain a sharp
estimate about dual quermassintegrals of any origin-symmetric convex body.
In Section \ref{sec4},
we prove Theorems \ref{thm1} and \ref{thm2} by a variational method.
Theorems \ref{thm3} and \ref{thm4} will be proved
in Sections \ref{sec6} and \ref{sec5} respectively.

\section{Preliminaries}
\label{sec2}

In this section we introduce some notations and preliminary results
about convex bodies and dual curvature measures.
The reader is referred to the book \cite{Schneider.2014}
and the article \cite{LYZ.Adv.329-2018.85}
for a comprehensive introduction on the background.

Let $\R^n$ be the $n$-dimensional Euclidean space, and $\uS$ the unit sphere.
A non-empty set $Q\subset\R^n$ is called \emph{star-shaped} with respect to the origin
if the line segment joining any point of $Q$ to the origin is completely
contained in $Q$.
For a compact star-shaped set $Q$, the \emph{radial function} $\rho_Q$ is defined as
\begin{equation*}
  \rho_Q(u)=\max\set{\lambda : \lambda u\in Q},
  \quad u\in\uS.
\end{equation*}
A \emph{star body} in $\R^n$ is a compact star-shaped subset
with respect to the origin,
which has a positive continuous radial function.
The set of all star bodies in $\R^n$ is denoted by $\mathcal{S}_o^n$.

A \emph{convex body} in $\R^n$ is a compact convex subset with non-empty interior.
Let $\K_o^n$ denote the class of convex bodies
containing the origin in their interiors, and
$\K_e^n$ the class of origin-symmetric convex bodies.
For a convex body $K$, its \emph{support function} $h_K$ is given by
\begin{equation*}
  h_K(x) =\max \left\{ \xi\cdot x : \,\xi\in K\right\},
  \quad x\in \uS.
\end{equation*}
Here ``$\cdot$'' denotes the inner product in the Euclidean space $\R^n$.
Note that $\K_o^n\subset\mathcal{S}_o^n$.
For any $K\in\K_o^n$, there is
\begin{equation} \label{eq:102}
  \frac{1}{\rho_K(u)} 
  = \max_{x\in\uS} \frac{u\cdot x}{h_K(x)},
  \quad u\in\uS.
\end{equation}

It is well known that a convex body is uniquely determined by its support
function, and the convergence of a sequence of convex bodies is equivalent to
the uniform convergence of the corresponding support functions on $\uS$.
The Blaschke selection theorem says that every bounded sequence of convex bodies
has a subsequence that converges to a compact convex subset.

Given $q\in\R$ and $Q\in\mathcal{S}_o^n$,
for any $K\in\K_o^n$, its \emph{generalized $q$-th dual curvature measure}
is defined as
\begin{equation}\label{gdCm}
  \widetilde{C}_q(K,Q,\eta)
  = \frac{1}{n} \int_{\boldsymbol{\alpha}_K^*(\eta)} \rho_K^q(u) \rho_Q^{n-q}(u) \dd u,
\end{equation}
where $\eta\subset\uS$ is any Borel subset, and
$\boldsymbol{\alpha}_K^*$ is the \emph{reverse radial Gauss image} given by
\begin{equation*}
  \boldsymbol{\alpha}_K^*(\eta)
  =\set{u\in\uS : \rho_K(u)u\in\nu_K^{-1}(\eta)}.
\end{equation*}
Here $\nu_K^{-1}$ is the inverse Gauss map of $K$.
From this definition, one can check that
\begin{equation} \label{eq:106}
  \int_{\uS} g(x) \dd\widetilde{C}_q(K,Q,x) 
  = \frac{1}{n}\int_{\uS}g(\nu_K(\rho_K(u)u))\rho_K^q(u)\rho_Q^{n-q}(u)\dd u
\end{equation} 
for any bounded Borel function $g$ on $\uS$.
We define the \emph{$q$-th dual mixed volume} $\widetilde{V}_q(K,Q)$ as that
\begin{equation}\label{dmV}
  \widetilde{V}_q(K,Q)=\frac{1}{n}\int_{\uS}\rho_K^q(u)\rho_Q^{n-q}(u)\dd u.
\end{equation}
Obviously, for any $\lambda>0$, there is
\begin{equation} \label{eq:82}
  \CTq(\lambda K,Q,\cdot) = \lambda^q \CTq(K,Q,\cdot),
  \qquad 
  \VTq(\lambda K,Q) = \lambda^q \VTq(K,Q).
\end{equation}
For $K_1\subset K_2$, we have the following monotonicity:
\begin{align}
  \VTq(K_1,Q) &\leq \VTq(K_2,Q),
                \quad \text{when } q>0,
                \label{eq:88} \\
  \VTq(K_1,Q) &\geq \VTq(K_2,Q),
                \quad \text{when } q<0.
                \label{eq:93}
\end{align}
The \emph{dual mixed entropy} $\widetilde{E}(K,Q)$ is defined as
\begin{equation} \label{dmEnt}
  \widetilde{E}(K,Q)
  = \frac{1}{n} \int_{\uS}
  \log\left( \frac{\rho_K(u)}{\rho_Q(u)} \right)
  \rho_Q^n(u) \dd u.
\end{equation}

Denote the set of positive continuous functions on $\uS$ by $C^+(\uS)$,
and the set of positive continuous even functions on $\uS$ by $C_e^+(\uS)$.
For $g\in C^+(\uS)$, the \emph{Alexandrov body associated with $g$} is defined by
\begin{equation*}
  K_g := \bigcap_{x\in \uS} \set{\xi\in\R^n \, :\, \xi \cdot x\leq g(x)}.
\end{equation*}
One can see that $K_g$ is a bounded convex body and $K_g\in\K_o^n$.
Note that
\begin{equation*}
  h_{K_g}(x) \leq g(x), \quad \forall\, x\in \uS.
\end{equation*}

The following variational formula was obtained in
\cite[Theorem 6.2]{LYZ.Adv.329-2018.85}.

\begin{lemma} \label{lemVari}
Let $\{g_t\}_{t\in(-\epsilon, \epsilon);\,\epsilon>0}$ be a family of positive
continuous functions on $\uS$.
If there is a continuous function $\varphi$ on $\uS$ such that
\begin{equation*}
  \lim_{t\to 0} \frac{g_t -g_0}{t}=\varphi
  \quad \text{uniformly on } \uS,
\end{equation*}
then for $Q\in\mathcal{S}_o^n$ we have that
\begin{equation*}
  \lim_{t\to 0} \frac{\widetilde{E}(K_{g_t},Q) -\widetilde{E}(K_{g_0},Q)}{t}
  = \int_{\uS} \varphi h_{K_{g_0}}^{-1} \dd\widetilde{C}_0(K_{g_0},Q),
\end{equation*}
and that for $q\neq0$,
\begin{equation*}
  \lim_{t\to 0} \frac{\widetilde{V}_q(K_{g_t},Q) -\widetilde{V}_q(K_{g_0},Q)}{t}
  = q\int_{\uS} \varphi h_{K_{g_0}}^{-1} \dd\widetilde{C}_q(K_{g_0},Q),
\end{equation*}
where $K_{g_t}$ is the Alexandrov body associated with $g_t$, and $h_{K_{g_0}}$
is the support function of $K_{g_0}$.
\end{lemma}

For $Q\in\mathcal{S}_o^n$, denote its volume by $\VOL(Q)$.
For a finite measure $\mu$ on $\uS$, write $|\mu|=\mu(\uS)$.
We use $\omega_{n-1}$ and $\kappa_n$ to denote the surface area and the volume
of the unit ball in $\R^n$ respectively.

\section{An integral estimate}
\label{sec3}

In this section, we prove the integral estimate Lemma \ref{lemInt},
which will be used to estimate dual quermassintegrals in the next section,
and may be of interest in its own right.

Let $A \in G\!L(n)$ be any diagonal matrix given by
\begin{equation}\label{A}
  A=\DIAG\left( s_{1}, \cdots, s_{n} \right)
  \ \text{ with }\ 
  s_{1}\ge \cdots \ge s_{n}>0.
\end{equation}


\begin{lemma}\label{lemInt}
For any positive number $\alpha>0$, we have
\begin{equation}\label{int}
  \int_{\uS} \frac{\dd x}{|Ax|^{\alpha}}
  \approx \begin{cases}
    1/(s_1\cdots s_{n}s_n^{\alpha-n}),
    & \text{when } \alpha\geq n, \\
    1/(s_1\cdots s_{\lceil \alpha \rceil}s_{\lceil \alpha \rceil}^{\alpha-{\lceil \alpha \rceil}}),
    & \text{when non-integer } \alpha<n, \\
    (1+\log(s_{\alpha}/s_{\alpha+1}))/(s_1\cdots s_{\alpha}),
    & \text{when integer } \alpha<n.
  \end{cases}
\end{equation}  
Here $\lceil\alpha\rceil$ is the smallest integer that is greater than or equal
to $\alpha$, 
and ``$\approx$'' means the ratio of the two sides has positive upper and lower
bounds depending only on $n$ and $\alpha$.
\end{lemma}

When Jian-Lu-Wang were studying the centroaffine Minkowski problem, a broader
class of integrals, including \eqref{int}, had already been studied in detail
\cite[Lemma 4.1]{JLW.JFA.274-2018.826}.
In fact, \eqref{int} is suggested in their long proof.
When only proving \eqref{int}, a very small part of their proof is just needed.
For readers' convenience, we here provide a self-contained proof following
\cite{JLW.JFA.274-2018.826}.

One main technique is the following dimension-reducing formula.

\begin{lemma}[Dimension-reducing formula] \label{lemL}
Assume $m\geq1$ is a positive integer,  
and $B$ is an $m$-order positive definite diagonal matrix.
For any $\beta\in(0,m)$, we have 
\begin{equation}\label{@eq:L}
  \int_{\mathbb{S}^{m-1}} \frac{\dd x}{|Bx|^{\beta}} 
  \approx
  \int_{\mathbb{S}^{\lfloor\beta\rfloor}}
  \frac{\dd y}{|B_{\scriptscriptstyle 1+\lfloor\beta\rfloor}y|^{\beta}},
\end{equation}
where $\lfloor\beta\rfloor$ is the largest integer that is less than or equal to $\beta$,
and $B_{ 1+\lfloor\beta\rfloor}$ is a diagonal matrix whose diagonal entries are
the top $1+\lfloor\beta\rfloor$ largest diagonal entries of $B$. 
\end{lemma}

\begin{proof}{}
Write $l=1+\lfloor \beta \rfloor$.
Then $0\leq l-1\leq\beta<l\leq m$.
There is nothing to prove if $l=m$.

Now assume $l\leq m-1$.
Then $m\geq2$ and $\beta\in(0,m-1)$.
Without loss of generality, assume $B$ is given as
\begin{equation*}
  B=\DIAG( s_{1}, \cdots, s_{m})
  \ \text{ with }\ 
  s_{1}\ge \cdots \ge s_{m}>0.
\end{equation*}
For simplicity, we write
\begin{equation*} 
  u=(x_1, \cdots, x_l), \quad
  v=(x_{l+1}, \cdots, x_{m}), \quad
  N=\DIAG( s_{l+1}, \cdots, s_{m}).
\end{equation*}
Then $x=(u,v)$ and $Bx=(B_lu,Nv)$.

By the coarea formula, we have for any $0\leq\delta<1$ that
\begin{equation}\label{@eq:jlw}
  \begin{split}
    \int_{\set{x\in\mathbb{S}^{m-1}:\,\delta\leq|u|\leq1}} \frac{\dd x}{|B_lu|^\beta}
    &=\int_{\delta\leq|u|\leq1} \frac{\dd u}{\lambda(u)} \int_{|v|=\lambda(u)} \frac{\dd\sigma(v)}{|B_lu|^{\beta}} \\
    &=\omega_{m-l-1} \int_{\delta\leq|u|\leq1} \lambda(u)^{m-l-2} \frac{\dd u}{|B_lu|^{\beta}} \\
    &=\omega_{m-l-1} \int_{\delta}^1 \lambda(r)^{m-l-2} \dd r \int_{|u|=r} \frac{\dd\sigma(u)}{|B_lu|^\beta} \\
    &=\omega_{m-l-1} \int_{\delta}^1 r^{l-1-\beta} \lambda(r)^{m-l-2} \dd r
    \int_{|y|=1} \frac{\dd\sigma(y)}{|B_ly|^\beta},
  \end{split}
\end{equation}
where $\lambda(u)=\sqrt{1-|u|^2}$.
Letting $\delta=0$, \eqref{@eq:jlw} becomes into
\begin{equation}\label{@eq:int}
  \int_{\mathbb{S}^{m-1}} \frac{\dd x}{|B_lu|^\beta}
  =C_{m,\beta}
  \int_{\mathbb{S}^{l-1}} \frac{\dd y}{|B_ly|^\beta}.
\end{equation}
Letting $\delta=1/2$, \eqref{@eq:jlw} becomes into
\begin{equation}\label{@eq:IS*1}
  \int_{S_*} \frac{\dd x}{|B_lu|^\beta}
  =C_{m,\beta}
  \int_{\mathbb{S}^{l-1}} \frac{\dd y}{|B_ly|^\beta},
\end{equation}
where $S_*=\set{x\in\mathbb{S}^{m-1} : 1/2\leq|u|\leq1}$.

Observe that for any $x\in S_*$, $|u|\geq1/2$, then $|v|\leq\sqrt{3}/2$.
Therefore, 
\begin{equation*}
  |Nv|\leq s_l|v|\leq \sqrt{3}s_l/2\leq\sqrt{3}|B_lu|,
\end{equation*}
implying that $|Bx|\leq2|B_lu|$ on $S_*$.
Hence, with \eqref{@eq:IS*1}, we have 
\begin{equation}\label{eq:30}
  \int_{\mathbb{S}^{m-1}} \frac{\dd x}{|Bx|^{\beta}} 
  \geq \int_{S_*} \frac{\dd x}{|Bx|^{\beta}} 
  \geq \frac{1}{2^{\beta}} \int_{S_*} \frac{\dd x}{|B_lu|^{\beta}} 
  =C_{m,\beta} \int_{\mathbb{S}^{l-1}} \frac{\dd y}{|B_ly|^{\beta}}.
\end{equation}
On the other hand, by \eqref{@eq:int}, there is
\begin{equation*}
  \int_{\mathbb{S}^{m-1}} \frac{\dd x}{|Bx|^{\beta}} 
  \leq \int_{\mathbb{S}^{m-1}} \frac{\dd x}{|B_lu|^{\beta}} 
  =C_{m,\beta} \int_{\mathbb{S}^{l-1}} \frac{\dd y}{|B_ly|^{\beta}},
\end{equation*} 
which together with \eqref{eq:30} yields the conclusion \eqref{@eq:L}.
\end{proof}

As a special case of the variable substitution formula 
\cite[Lemma 2.2]{Lu.SCM.61-2018.511},
we have 

\begin{lemma}[Power-reducing formula] \label{lemPRF}
Assume $m\geq1$ is a positive integer,  
and $B$ is an $m$-order invertible matrix.
For any $\gamma\in\R$, we have 
\begin{equation} \label{scm2.2}
  \int_{\mathbb{S}^{m-1}} \frac{\dd x}{|Bx|^{\gamma}}
  = \frac{1}{|\det B|}
  \int_{\mathbb{S}^{m-1}} \frac{\dd x}{|B^{-1}x|^{m-\gamma}}.
\end{equation}
\end{lemma}

With these two Lemmas \ref{lemL} and \ref{lemPRF}, one can easily prove Lemma \ref{lemInt}.

\begin{proof}[\textbf{Proof of Lemma \ref{lemInt}}]
\textup{(a)}  
When $\alpha\geq n$.
By Lemma \ref{lemPRF}, there is
\begin{equation*}
  \begin{split}
    \int_{\uS} \frac{\dd x}{|Ax|^{\alpha}}
    &= \frac{1}{\det A} \int_{\uS} |A^{-1}x|^{\alpha-n} \dd x \\
    &\approx \frac{1}{\det A}
    \int_{\uS} \left(\abs{s_1^{-1}x_1}^{\alpha-n} +\cdots+ \abs{s_n^{-1}x_n}^{\alpha-n} \right) \dd x
    \approx \frac{1}{\det A}s_n^{n-\alpha}.
  \end{split}
\end{equation*}

\textup{(b)}  
When $\alpha< n$.
Applying Lemma \ref{lemL}, we have with $l=1+\lfloor\alpha\rfloor$ that
\begin{equation}\label{eq:71}
  \int_{\uS} \frac{\dd x}{|Ax|^{\alpha}} 
  \approx \int_{\mathbb{S}^{l-1}} \frac{\dd y}{|A_ly|^{\alpha}}
  = \frac{1}{\det A_l} \int_{\mathbb{S}^{l-1}} \frac{\dd y}{|Py|^{l-\alpha}},
\end{equation}
where the last equality is due to Lemma \ref{lemPRF}, and
$P=(A_l)^{-1}=\DIAG(s_1^{-1},\cdots,s_l^{-1})$.

If $\alpha$ is a non-integer, there is
$l=\lceil \alpha \rceil$ and $l-\alpha\in(0,1)$.
By Lemma \ref{lemL} again,
\begin{equation*}
  \int_{\mathbb{S}^{l-1}} \frac{\dd y}{|Py|^{l-\alpha}}
  \approx \int_{\mathbb{S}^{0}} \frac{\dd z}{|P_1z|^{l-\alpha}} 
  = \frac{2}{s_l^{\alpha-l}},
\end{equation*}
which together with \eqref{eq:71} yields the conclusion \eqref{int} for any
non-integer $\alpha<n$.

If $\alpha$ is an integer, there is $l=\alpha+1$.
By Lemma \ref{lemL} again, we obtain
\begin{multline*}
  \int_{\mathbb{S}^{l-1}} \frac{\dd y}{|Py|^{l-\alpha}}
  = \int_{\mathbb{S}^{\alpha}} \frac{\dd y}{|Py|}
  \approx \int_{\mathbb{S}^{1}} \frac{\dd z}{|P_2z|} 
  = 4\int_{0}^{\frac{\pi}{2}}
  \frac{\dd t}{\sqrt{s_l^{-2}\sin^2t+s_{\alpha}^{-2}\cos^2t}} \\
  \approx
  \int_{0}^{1} \frac{\dd t}{s_l^{-1}t+s_{\alpha}^{-1}}
  +\int_{1}^{\frac{\pi}{2}} \frac{\dd t}{s_l^{-1}} 
  \approx
  s_l\left(1+\log\frac{s_\alpha}{s_l}\right),
\end{multline*}
which together with \eqref{eq:71} yields the conclusion \eqref{int} for any
integer $\alpha<n$.
\end{proof}

\section{The case $0<q<n$}
\label{sec4}

We shall prove Theorems \ref{thm1} and \ref{thm2} in this section.
In fact, we mainly prove the following existence lemma.

\begin{lemma}\label{lemq0n}
Assume $0<q<n$, and $Q$ is an origin-symmetric star body in $\R^n$.
If $\mu$ is a finite even Borel measure on $\uS$ satisfying
the following $q$-th subspace mass inequality:
\begin{equation} \label{sMassIneq}
  \frac{\mu(\uS\cap\xi_i)}{\mu(\uS)}<
  \min\set{\frac{i}{q},1}
\end{equation}
for any proper $i$-dimensional subspace $\xi_i\subset\R^n$ with
$i=1,\cdots,n-1$,
then there exists an origin-symmetric convex body $K$ in $\R^n$ satisfying
\begin{equation}\label{eq:113}
  \widetilde{C}_q(K,Q,\cdot)=\mu.
\end{equation}
\end{lemma}

Assuming this lemma for the moment,
we can then easily prove Theorems \ref{thm1} and \ref{thm2}.
In fact,
Theorem \ref{thm1} is just Lemma \ref{lemq0n} with $1<q<n$.

\begin{proof}[\textbf{Proof of Theorem \ref{thm2}}]
The necessity is obvious.
For the sufficiency, if
$\mu$ is not concentrated on any great sub-sphere of $\uS$,
then 
it satisfies the $q$-th subspace mass inequality \eqref{sMassIneq} for $0<q\leq1$.
By Lemma \ref{lemq0n}, 
there exists an origin-symmetric convex body $K$ in $\R^n$ satisfying
$\widetilde{C}_q(K,Q,\cdot)=\mu$.
Therefore, Theorem \ref{thm2} is true.
\end{proof}

Therefore, in the rest of this section, it suffices to prove Lemma \ref{lemq0n}.
Consider the following minimizing problem:
\begin{equation}\label{minP}
  \inf\set{ J[g] : g\in C_e^+(\uS)},
\end{equation} 
where
\begin{equation} \label{J}
  J[g]=
  \frac{1}{|\mu|} \int_{\uS} \log g\, \dd\mu 
  -\frac{1}{q}\log\widetilde{V}_q(K_g,Q).
\end{equation}
Here we recall that $K_g$ is the Alexandrov body associated with $g$,
and $\VTq$ is the $q$-th dual mixed volume given in \eqref{dmV}.
In the following, we will prove \eqref{minP} has a solution $h$, and
a multiple of $h$ is a solution to Eq. \eqref{eq:113}.

\subsection{An entropy-type integral estimate}

We first deal with the part $\frac{1}{|\mu|} \int_{\uS} \log g\, \dd\mu$.
It is based on an appropriate spherical partition,
which was introduced in \cite{BLYZ.JAMS.26-2013.831};
see also \cite{BLY+.Adv.356-2019.106805}.

\begin{lemma}\label{lemEnt}
Assume $0<q<n$, and $\mu$ is a finite even Borel measure on $\uS$ satisfying
the $q$-th subspace mass inequality \eqref{sMassIneq}.
Then for any sequence of origin-centered ellipsoids $\set{E_k}$
with lengths of the semi-axes $b_{1k}\leq\cdots\leq b_{nk}$, 
there exists a subsequence $\set{E_{k'}}$,
two small positive numbers $\epsilon_0$, $\delta_0$,
and an integer $k_0$, such that for any $k'\geq k_0$,
\begin{equation} \label{eq:1}
  \frac{1}{|\mu|} \int_{\uS} \log h_{E_{k'}} \dd\mu
  \geq
  \log\left(
    \frac{\delta_0}{2} 
    b_{nk'}^{\epsilon_0}
    b_{1k'}^{-\epsilon_0}
    b_{\lceil q \rceil k'}^{(q-\lceil q \rceil)/q}
    \prod_{i=1}^{\lceil q \rceil} b_{ik'}^{1/q}
  \right).
\end{equation} 
\end{lemma}

\begin{proof}{}
For each ellipsoid $E_k$, there exists an orthogonal matrix $P_k$ such that
\begin{equation}\label{eq:8}
  E_k=\set{y\in\R^n : |B_k^{-1}P_ky|\leq1},
\end{equation}  
where $B_k=\DIAG(b_{1k},\cdots,b_{nk})$.
Without loss of generality, we can assume that $P_k$ tends to some orthogonal
matrix $P$ as $k\to+\infty$.

For simplicity, write
$P=(\eta_1,\cdots,\eta_n)^T$ and
the identity matrix $I_n=(e_1,\cdots,e_n)^T$.
For each $\delta\in(0,\frac{1}{\sqrt{n}})$, we define
\begin{equation*}
  \Omega_{n\delta}=\set{x\in\uS : |\eta_n^Tx|>\delta}, 
\end{equation*}
and for $i=n-1,\cdots,1$, define
\begin{equation*}
  \Omega_{i\delta}=\set{x\in\uS : |\eta_n^Tx|\leq\delta, \cdots, |\eta_{i+1}^Tx|\leq\delta, |\eta_i^Tx|>\delta}.
\end{equation*}
Obviously, they are mutually disjoint subsets of $\uS$.
Moreover,
\begin{equation*}
  \uS\setminus\left( \cup_{i=1}^n \Omega_{i\delta} \right)
  = \set{x\in\uS : |\eta_n^Tx|\leq\delta, \cdots, |\eta_{1}^Tx|\leq\delta}
\end{equation*}
is an empty set due to $0<\delta<\frac{1}{\sqrt{n}}$.
Therefore, $\set{\Omega_{1\delta},\cdots,\Omega_{n\delta}}$ is a partition of $\uS$.

To determine the limits of $\Omega_{i\delta}$ as $\delta\to0^+$, we construct
\begin{align*}
  \Omega_{i\delta}'&=\set{x\in\uS : |\eta_n^Tx|=0, \cdots, |\eta_{i+1}^Tx|=0, |\eta_i^Tx|>\delta}, \\
  \Omega_{i\delta}''&=\set{x\in\uS : |\eta_n^Tx|\leq\delta, \cdots, |\eta_{i+1}^Tx|\leq\delta, |\eta_i^Tx|\neq0},
\end{align*}
satisfying
\begin{equation}\label{eq:2}
  \Omega_{i\delta}'\subset\Omega_{i\delta}\subset\Omega_{i\delta}''.
\end{equation}
Observe that, as $\delta\searrow0^+$, $\Omega_{i\delta}'$ is increasing
and $\Omega_{i\delta}''$ is decreasing.
Both of them have the same limit
\begin{equation*}
  \set{x\in\uS : |\eta_n^Tx|=0, \cdots, |\eta_{i+1}^Tx|=0, |\eta_i^Tx|\neq0},
\end{equation*}
which can be written as $\uS\cap (\xi_i\setminus\xi_{i-1})$,
if we define 
\begin{equation*}
  \xi_0=\set{0}
  \ \text{ and }\ 
  \xi_i=\SPAN\set{\eta_1,\cdots,\eta_i}
  \ \text{ for }\ 
  i=1,\cdots,n.
\end{equation*}
Now recalling \eqref{eq:2}, we obtain that
\begin{equation*}
  \lim_{\delta\to0^+} \Omega_{i\delta}
  =\uS\cap (\xi_i\setminus\xi_{i-1}),
\end{equation*}
implying
\begin{equation*}
  \lim_{\delta\to0^+} \mu(\Omega_{i\delta})
  =\mu(\uS\cap (\xi_i\setminus\xi_{i-1}))
  =\mu(\uS\cap \xi_i) -\mu(\uS\cap \xi_{i-1}).
\end{equation*}
Hence, for each $i=n,\cdots,2$, we have
\begin{equation}\label{eq:5}
  \lim_{\delta\to0^+}
  (\mu(\Omega_{n\delta}) +\cdots +\mu(\Omega_{i\delta}))
  =\mu(\uS)-\mu(\uS\cap \xi_{i-1}).
\end{equation}
Then
\begin{equation*}
  \begin{split}
    \lim_{\delta\to0^+}
    \frac{\mu(\Omega_{n\delta}) +\cdots +\mu(\Omega_{i\delta})}{\mu(\uS)}
    &=1-\frac{\mu(\uS\cap \xi_{i-1})}{\mu(\uS)} \\
    &>1-\min\set{\frac{i-1}{q},1} \\
    &=\max\set{\frac{q+1-i}{q},0},
  \end{split}
\end{equation*}
where the $q$-th subspace mass inequality \eqref{sMassIneq} has been used.
Thus, there exist two small positive numbers $\epsilon_0$ and $\delta_0$, such that
\begin{equation}\label{eq:6}
  \frac{\mu(\Omega_{n\delta_0}) +\cdots +\mu(\Omega_{i\delta_0})}{\mu(\uS)}
  >\epsilon_0+ \max\set{\frac{q+1-i}{q},0},
  \quad i=2,\cdots,n.
\end{equation}

Recalling $P_k\to P$ as $k\to+\infty$, one can find a large $k_0$, such that
\begin{equation*}
  \norm{P_k-P}<\frac{\delta_0}{2},
  \quad \forall\,k\geq k_0.
\end{equation*}
We now estimate $h_{E_k}$ in $\Omega_{i\delta_0}$ for every $i=1,\cdots,n$.
Here and in the following proof, we always assume that $k\geq k_0$.
Recalling \eqref{eq:8}, since
\begin{equation*}
  |B_k^{-1}P_kb_{ik}P_k^Te_i|=|e_i|=1,
\end{equation*}
we see $\pm b_{ik}P_k^Te_i\in E_k$.
By the definition of support function, we have
\begin{equation}\label{eq:72}
  h_{E_k}(x) \geq b_{ik} |e_i^TP_kx|.
\end{equation}
Recalling
$|\eta_i^Tx|>\delta_0$ for $x\in\Omega_{i\delta_0}$, there is
\begin{equation*}
  \begin{split}
    |e_i^TP_kx|
    &\geq |e_i^TPx|-|e_i^T(P_k-P)x| \\
    &\geq |\eta_i^Tx|-\norm{P_k-P} \\
    &>\frac{\delta_0}{2},
  \end{split}
\end{equation*}
which together with \eqref{eq:72} yields
\begin{equation} \label{eq:73}
  h_{E_k}(x) > \frac{\delta_0}{2} b_{ik},
  \quad \forall\,x\in\Omega_{i\delta_0}.
\end{equation}
Recalling that 
$\set{\Omega_{1\delta_0},\cdots,\Omega_{n\delta_0}}$ is a partition of $\uS$,
we have
\begin{equation} \label{eq:74}
  \begin{split}
    \Lambda_k:=
    \frac{1}{|\mu|} \int_{\uS} \log h_{E_k}(x) \dd\mu(x) 
    &=
    \frac{1}{|\mu|} \sum_{i=1}^n \int_{\Omega_{i\delta_0}} \log h_{E_k}(x) \dd\mu(x) \\
    &\geq \frac{1}{|\mu|} \sum_{i=1}^n
    \mu(\Omega_{i\delta_0})
    \log\left(\frac{\delta_0}{2} b_{ik} \right).
  \end{split}
\end{equation}
Now for simplicity, denote
\begin{equation*}
  m_i:= \frac{\mu(\Omega_{i\delta_0})}{|\mu|},
  \quad i=1,\cdots,n.
\end{equation*}
Then, \eqref{eq:74} becomes
\begin{equation}\label{eq:75} 
  \Lambda_k
  \geq \sum_{i=1}^n m_i \log\left(\frac{\delta_0}{2} b_{ik} \right)
  = \log \left(
    \frac{\delta_0}{2}
    b_{nk}^{m_n}\cdots b_{1k}^{m_1}
  \right),
\end{equation}
where the fact $m_n+\cdots+m_1=1$ has been used.

Recall the estimate \eqref{eq:6}, which says
\begin{equation}\label{eq:7}
  m_n+\cdots+m_i >\epsilon_0+\tau_i,
  \quad i=2,\cdots,n,
\end{equation}
where
\begin{equation}\label{eq:76}
  \tau_i
  = \max\set{\frac{q+1-i}{q},0}
  =\begin{cases}
    (q+1-i)/q, & \text{when } 2\leq i\leq \lceil q \rceil, \\
    0, & \text{when } \lceil q \rceil+1\leq i\leq n.
  \end{cases}
\end{equation}
On account of \eqref{eq:7} and $0<b_{1k}\leq\cdots\leq b_{nk}$, we have the
following computations:
\begin{equation}\label{eq:77}
  \begin{split}
    b_{nk}^{m_n}\cdots b_{1k}^{m_1}
    &=
    \left(\frac{b_{nk}}{b_{n-1;k}}\right)^{m_n}
    \cdots
    \left(\frac{b_{ik}}{b_{i-1;k}}\right)^{m_n+\cdots+m_i}
    \cdots
    \left(\frac{b_{2k}}{b_{1k}}\right)^{m_n+\cdots+m_2}
    b_{1k}^{m_n+\cdots+m_1} \\
    &\geq
    \left(\frac{b_{nk}}{b_{n-1;k}}\right)^{\epsilon_0+\tau_n}
    \cdots
    \left(\frac{b_{ik}}{b_{i-1;k}}\right)^{\epsilon_0+\tau_i}
    \cdots
    \left(\frac{b_{2k}}{b_{1k}}\right)^{\epsilon_0+\tau_2}
    b_{1k} \\
    &=
    b_{nk}^{\epsilon_0+\tau_n}
    b_{n-1;k}^{\tau_{n-1}-\tau_n}
    \cdots
    b_{2k}^{\tau_2-\tau_3}
    b_{1k}^{1-\epsilon_0-\tau_2}.
  \end{split}
\end{equation}
When $q\in(0,1]$, we have $\lceil q \rceil=1$.
From \eqref{eq:76}, there is $\tau_i=0$ for $i=2,\cdots,n$.
Hence, the estimate \eqref{eq:77} is simplified as
\begin{equation} \label{eq:78}
  b_{nk}^{m_n}\cdots b_{1k}^{m_1}
  \geq
  b_{nk}^{\epsilon_0}
  b_{1k}^{1-\epsilon_0}.
\end{equation}
When $q\in(n-1,n)$, we have $\lceil q \rceil=n$.
From \eqref{eq:76}, there is $\tau_i=(q+1-i)/q$ for $i=2,\cdots,n$,
implying $\tau_i-\tau_{i+1}=1/q$ for $i=2,\cdots,n-1$.
Therefore, \eqref{eq:77} is simplified as
\begin{equation}\label{eq:79}
  b_{nk}^{m_n}\cdots b_{1k}^{m_1}
  \geq
  b_{nk}^{\epsilon_0+(q+1-n)/q}
  b_{1k}^{-\epsilon_0+1/q}
  \prod_{i=2}^{n-1} b_{ik}^{1/q}.
\end{equation}
When $q\in(1,n-1]$, we have $2\leq\lceil q \rceil\leq n-1$.
By \eqref{eq:76}, one can see that
\begin{equation*}
  \tau_i-\tau_{i+1}=
  \begin{cases}
    1/q, & \text{when } 2\leq i\leq \lceil q \rceil-1, \\
    (q+1-\lceil q \rceil)/q, & \text{when }  i= \lceil q \rceil, \\
    0, & \text{when } \lceil q \rceil+1\leq i\leq n-1,
  \end{cases}
\end{equation*}
which together with $\tau_2=\frac{q-1}{q}$ and $\tau_n=0$
implies that \eqref{eq:77} reads
\begin{equation}\label{eq:80}
  \begin{split}
    b_{nk}^{m_n}\cdots b_{1k}^{m_1}
    &\geq
    b_{nk}^{\epsilon_0+\tau_n}
    b_{1k}^{1-\epsilon_0-\tau_2}
    \prod_{i=2}^{\lceil q \rceil} b_{ik}^{\tau_i-\tau_{i+1}} \\
    &=
    b_{nk}^{\epsilon_0}
    b_{1k}^{-\epsilon_0+1/q}
    b_{\lceil q \rceil k}^{(q-\lceil q \rceil)/q}
    \prod_{i=2}^{\lceil q \rceil} b_{ik}^{1/q}.
  \end{split}
\end{equation}
Note that \eqref{eq:78}, \eqref{eq:79} and \eqref{eq:80} can be unified into the
following form: 
\begin{equation} \label{eq:81}
  b_{nk}^{m_n}\cdots b_{1k}^{m_1}
  \geq
  b_{nk}^{\epsilon_0}
  b_{1k}^{-\epsilon_0}
  b_{\lceil q \rceil k}^{(q-\lceil q \rceil)/q}
  \prod_{i=1}^{\lceil q \rceil} b_{ik}^{1/q}
  \quad \text{when }\  q\in(0,n).
\end{equation}

Now, inserting \eqref{eq:81} into \eqref{eq:75},
we obtain the conclusion \eqref{eq:1}. 
\end{proof}

\subsection{A sharp estimate about dual quermassintegrals}

Based on Lemma \ref{lemInt}, we can easily have a sharp estimate about $\VTq$
for origin-centered ellipsoids.

\begin{lemma}\label{lem314}
Let $Q$ be a star body in $\R^n$.
For any origin-centered ellipsoid $E\subset\R^n$  
with lengths of the semi-axes $b_{1}\leq\cdots\leq b_{n}$, 
we have 
\begin{equation} \label{eq:89}
  \VTq(E,Q)\approx 
  \begin{cases}
    b_1\cdots b_{n}b_n^{q-n},
    & \text{when } q\in[n,\infty), \\
    b_1\cdots b_{\lceil q \rceil}b_{\lceil q \rceil}^{q-{\lceil q \rceil}},
    & \text{when non-integer } q\in(0,n), \\
    b_1\cdots b_{q} (1+\log(b_{q+1}/b_{q})),
    & \text{when integer } q\in(0,n),
  \end{cases}
\end{equation}
where the ``$\approx$'' means the ratio of the two sides has positive upper and lower
bounds depending only on $n$, $q$, $\min\rho_Q$ and $\max\rho_Q$.
\end{lemma}

\begin{proof}{}
Note that $E$ can be expressed as   
\begin{equation*}
  E=\set{y\in\R^n : |APy|\leq1},
\end{equation*}  
where $P$ is some orthogonal matrix of order $n$, and
\begin{equation*}
  A=\DIAG(b_{1}^{-1},\cdots,b_{n}^{-1}).
\end{equation*}  
Then, $\rho_E(u)=\frac{1}{|APu|}$ for $u\in\uS$.
Thus, we have for $q>0$ that
\begin{equation*}
  \begin{split}
    \widetilde{V}_q(E,Q)
    &= \frac{1}{n}\int_{\uS}\rho_E^q(u)\rho_Q^{n-q}(u)\dd u \\
    &\approx \int_{\uS}\rho_E^q(u)\dd u \\
    &= \int_{\uS} \frac{\dd u}{|APu|^q} \\ 
    &= \int_{\uS} \frac{\dd u}{|Au|^q}.
  \end{split}
\end{equation*}
Now the conclusion \eqref{eq:89} follows directly from Lemma \ref{lemInt}. 
\end{proof}

\subsection{Existence of solutions to the minimizing problem}

We are now in position to prove that the minimizing problem \eqref{minP} has a
solution.

First note that $J$ defined in \eqref{J} is homogeneous of degree zero.
In fact, by \eqref{eq:82}, for $g\in C^+(\uS)$ and $\lambda>0$, there is
\begin{equation*}
  \VTq(K_{\lambda g},Q)
  = \VTq(\lambda K_{g},Q)
  = \lambda^q \VTq(K_{g},Q),
\end{equation*}
which leads to 
\begin{equation}\label{eq:4}
  J[\lambda g]=J[g].
\end{equation}

\begin{lemma}\label{lem300}
Under the assumptions of Lemma \ref{lemq0n}, the minimizing problem \eqref{minP}
has a solution $h$.
In addition, the solution $h$ is the support function of $K_h$. 
\end{lemma}

\begin{proof}{}
Assume $\set{g_k}\subset C_e^+(\uS)$ is a minimizing sequence of \eqref{minP}, namely
\begin{equation}\label{eq:83}
  J[g_k]\to
  \inf\set{ J[g] : g\in C_e^+(\uS)}
  \quad \text{as } k\to+\infty.
\end{equation}  
Since $g_k$ is even, the Alexandrov body $K_{g_k}$ is origin-symmetric.
Let $h_k$ be the support function of $K_{g_k}$.
Then $h_k\in C_e^+(\uS)$, $h_k\leq g_k$, and $K_{h_k}=K_{g_k}$.
Therefore,
\begin{equation*}
  \begin{split}
    J[h_k] &=
    \frac{1}{|\mu|} \int_{\uS} \log h_k \dd\mu
    -\frac{1}{q}\log\widetilde{V}_q(K_{h_k},Q) \\
    &\leq
    \frac{1}{|\mu|} \int_{\uS} \log g_k \dd\mu
    -\frac{1}{q}\log\widetilde{V}_q(K_{g_k},Q) \\
    &= J[g_k],
  \end{split} 
\end{equation*}
which together with \eqref{eq:83} implies that
\begin{equation}\label{eq:84}
  J[h_k]\to
  \inf\set{ J[g] : g\in C_e^+(\uS)}
  \quad \text{as } k\to+\infty.
\end{equation}  
Namely, $\set{h_k}$ is also a minimizing sequence of \eqref{minP}.
Recalling the zeroth-order homogeneity of $J$ given in \eqref{eq:4},
we can assume that $\max_{\uS}h_k=\sqrt{n}$ for every $k$.

Let $E_k$ be the maximum-volume ellipsoid of $K_{h_k}$.
Then $E_k$ is origin-centered and satisfies 
\begin{equation}\label{eq:3}
  E_k \subset K_{h_k} \subset \sqrt{n} E_k,
\end{equation}
implying that
\begin{equation} \label{eq:87}
  h_{E_k} \leq h_k \leq \sqrt{n}h_{E_k}. 
\end{equation}
Let $b_{1k}\leq\cdots\leq b_{nk}$ be the lengths of the semi-axes of $E_k$.
From \eqref{eq:87}, we have
\begin{equation*}
  \max h_{E_k}
  \leq
  \max h_k
  \leq
  \sqrt{n} \max h_{E_k},
\end{equation*}
which together with $\max h_{E_k}=b_{nk}$ and $\max h_k=\sqrt{n}$ implies that
\begin{equation} \label{eq:86}
  1\leq b_{nk}\leq \sqrt{n}
  \quad\text{for every }k.
\end{equation}

Recall $\VTq(\cdot,Q)$ is increasing when $q\in(0,n)$; see \eqref{eq:88}.
By virtue of \eqref{eq:87} and \eqref{eq:3}, we have
\begin{equation*}
  \begin{split}
    J[h_k]
    &=
    \frac{1}{|\mu|} \int_{\uS} \log h_k \dd\mu
    -\frac{1}{q}\log\widetilde{V}_q(K_{h_k},Q) \\
    &\geq
    \frac{1}{|\mu|} \int_{\uS} \log h_{E_k} \dd\mu
    -\frac{1}{q}\log\widetilde{V}_q(\sqrt{n}E_k,Q) \\
    &=
    \frac{1}{|\mu|} \int_{\uS} \log h_{E_k} \dd\mu
    -\frac{1}{q}\log\widetilde{V}_q(E_k,Q)
    -\log\sqrt{n},
  \end{split} 
\end{equation*}
where the last equality is due to \eqref{eq:82}.
Applying Lemma \ref{lemEnt} to the sequence $\set{E_k}$, 
there exists a subsequence $\set{E_{k'}}$,
two positive numbers $\epsilon_0$, $\delta_0$, and an integer $k_0$,
such that for any $k'\geq k_0$, the estimate \eqref{eq:1} holds.
Therefore,
\begin{equation} \label{eq:85}
  \begin{split}
    J[h_{k'}]
    &\geq
    \log\left(
      b_{nk'}^{\epsilon_0}
      b_{1k'}^{-\epsilon_0}
      b_{\lceil q \rceil k'}^{(q-\lceil q \rceil)/q}
      \prod_{i=1}^{\lceil q \rceil} b_{ik'}^{1/q}
    \right)
    -\frac{1}{q}\log\widetilde{V}_q(E_{k'},Q)
    +\log\left(\frac{\delta_0}{2\sqrt{n}} \right) \\
    &=
    \frac{1}{q} \log\left( \frac{
        b_{nk'}^{q\epsilon_0}
        b_{1k'}^{-q\epsilon_0}
        b_{\lceil q \rceil k'}^{q-\lceil q \rceil}
        \prod_{i=1}^{\lceil q \rceil} b_{ik'}
      }{\VTq(E_{k'},Q)} \right)
    +\log\left(\frac{\delta_0}{2\sqrt{n}} \right).
  \end{split} 
\end{equation}

Note that $\VTq(E_{k'},Q)$ can be estimated by Lemma \ref{lem314}.
When $q\in(0,n)$ is a non-integer,
\begin{equation*}
  \VTq(E_{k'},Q)
  \approx
  b_{1k'}\cdots b_{\lceil q \rceil k'}b_{\lceil q \rceil k'}^{q-{\lceil q \rceil}},
\end{equation*}
implying that \eqref{eq:85} can be reduced to
\begin{equation} \label{eq:90}
  J[h_{k'}] \geq
  \epsilon_0 \log\left(b_{nk'}/b_{1k'} \right)
  -C_0,
\end{equation}
where $C_0$ is a positive constant depending only on
$n$, $q$, $\min\rho_Q$, $\max\rho_Q$ and $\delta_0$.
When $q\in(0,n)$ is an integer, there is
\begin{equation*}
  \VTq(E_{k'},Q)
  \approx
  b_{1k'}\cdots b_{qk'} (1+\log(b_{q+1;k'}/b_{qk'})),
\end{equation*}
which together with $\lceil q \rceil=q$ implies
that \eqref{eq:85} can be reduced to
\begin{equation} \label{eq:91}
  \begin{split}
    J[h_{k'}]
    &\geq
    \frac{1}{q} \log\left( \frac{
        b_{nk'}^{q\epsilon_0}
        b_{1k'}^{-q\epsilon_0}
      }{1+\log(b_{q+1;k'}/b_{qk'})} \right)
    -C_0 \\
    &\geq
    \frac{1}{q} \log\left( \frac{
        b_{nk'}^{q\epsilon_0}
        b_{1k'}^{-q\epsilon_0}
      }{1+\log(b_{nk'}/b_{1k'})} \right)
    -C_0,
  \end{split} 
\end{equation}
where $C_0$ is again a positive constant independent of $k'$.

Recalling \eqref{eq:84}, which says as $k'\to+\infty$ that
\begin{equation*}
  J[h_{k'}]\to
  \inf\set{ J[g] : g\in C_e^+(\uS)}
  \leq J[1],
\end{equation*}  
where $J[1]= -\frac{1}{q}\log\widetilde{V}_{n-q}(Q)$ is a finite number.
Without loss of generality, we assume 
\begin{equation*}
  J[h_{k'}]<1+J[1], 
  \quad \forall\,k'\geq k_0.
\end{equation*}
Combining it with \eqref{eq:90} and \eqref{eq:91},
we obtain for each $q\in(0,n)$ that
\begin{equation*}
  \frac{b_{nk'}}{b_{1k'}}\leq C,
  \quad \forall\,k',
\end{equation*}
where $C$ is a positive constant independent of $k'$.
Now, by \eqref{eq:86}, we have
\begin{equation*}
  b_{1k'} \geq \frac{1}{C},
  \quad \forall\,k',
\end{equation*}
which together with \eqref{eq:87} implies
\begin{equation*} 
  \min_{\uS} h_{k'} \geq b_{1k'} \geq \frac{1}{C},
  \quad \forall\,k'.
\end{equation*}
Note that $h_{k'}$ is the support function of $K_{h_{k'}}$, 
and $\max_{\uS}h_{k'}=\sqrt{n}$.
Applying Blaschke selection theorem to $\set{h_{k'}}$,
there is a subsequence, still denoted by $\set{h_{k'}}$,
which uniformly converges to some support function $h$ on $\uS$.
Obviously, $\frac{1}{C}\leq h\leq \sqrt{n}$ on $\uS$,
namely, $h\in C_e^+(\uS)$.
Correspondingly, $K_{h_{k'}}$ converges to $K_h\in \mathcal{K}_e^n$
which is the convex body determined by $h$.
Recalling the definition of $J$ in \eqref{J},
there is $\lim_{k'\to+\infty}J[h_{k'}]=J[h]$.
By \eqref{eq:84} again, we have
\begin{equation*}
  J[h]=
  \inf\set{ J[g] : g\in C_e^+(\uS)}.
\end{equation*}  
Thus, $h$ is a solution to the minimizing problem \eqref{minP}.
The proof of this lemma is now completed.
\end{proof}

\subsection{Existence of solutions to the generalized dual Minkowski problem}

By virtue of the variational formula Lemma \ref{lemVari}, one can prove the
following lemma.

\begin{lemma}\label{lem317}
A multiple of the minimizer $h$ obtained in Lemma \ref{lem300}
solves Eq. \eqref{eq:113}.
\end{lemma}

\begin{proof}
Let $h$ be the solution obtained in Lemma \ref{lem300}. 
For any given continuous even function $\varphi\in C(\uS)$, let
\begin{equation*}
  g_t=h+t\varphi \quad \text{for small } t\in\R.
\end{equation*}
Since $h\in C^+_{e}(\uS)$, for $t$ sufficiently small $g_t\in C^+_{e}(\uS)$ as well.
By Lemma \ref{lemVari}, we have
\begin{equation} \label{eq:92}
  \frac{\dd}{\dd t} \widetilde{V}_q(K_{g_t},Q) \Big|_{t=0}
  = q\int_{\uS} \varphi h^{-1} \dd\widetilde{C}_q(K_{h},Q).
\end{equation}

Write $J(t)=J[g_t]$.
Then $J(0)=J[h]$.
Since $h$ is a minimizer of \eqref{minP}, there is
\begin{equation*}
  J(t)\geq J(0)
  \quad \text{for any small } t\in\R,
\end{equation*}
which together with \eqref{eq:92} and the definition of $J$ in \eqref{J} yields that
\begin{equation*}
  \begin{split}
    0&= \frac{\dd}{\dd t} J(t) \Big|_{t=0} \\
    &= \frac{\dd}{\dd t} \left(
      \frac{1}{|\mu|} \int_{\uS} \log g_t\, \dd\mu 
      -\frac{1}{q}\log\widetilde{V}_q(K_{g_t},Q)
    \right) \bigg|_{t=0} \\
    &=
    \frac{1}{|\mu|} \int_{\uS} \varphi h^{-1} \dd\mu 
    -\frac{1}{\widetilde{V}_q(K_{h},Q)}
    \int_{\uS} \varphi h^{-1} \dd\widetilde{C}_q(K_{h},Q).
  \end{split}
\end{equation*}
Note that this equality holds for arbitrary even function $\varphi$,
and that $\mu$, $\CTq(K_h,Q)$ are even Borel measures.
Therefore, we obtain
\begin{equation*}
  \frac{1}{|\mu|} \mu 
  =\frac{1}{\widetilde{V}_q(K_{h},Q)}
  \widetilde{C}_q(K_{h},Q).
\end{equation*}
Letting
\begin{equation*}
  c=\left(
    \frac{|\mu|}{\widetilde{V}_q(K_{h},Q)}
  \right)^{1/q},
\end{equation*}
and recalling \eqref{eq:82}, we have
\begin{equation*}
  \widetilde{C}_q(c K_{h},Q)
  =\mu. 
\end{equation*}
The proof of this lemma is completed.
\end{proof}

Now the proof of Lemma \ref{lemq0n} is completed.

\section{The case $q<0$}
\label{sec5}

In this section, we prove Theorem \ref{thm4}.
For $q<0$, we can consider the same minimizing problem as used for the case
$0<q<n$ in the previous section. 
Since $\VTq$ is easy to estimate when $q<0$,
one can drop the evenness assumption.
Therefore, we consider the following minimizing problem:
\begin{equation}\label{minP1}
  \inf\set{ J[g] : g\in C^+(\uS)},
\end{equation} 
where
\begin{equation*} 
  J[g]=
  \frac{1}{|\mu|} \int_{\uS} \log g\, \dd\mu 
  -\frac{1}{q}\log\widetilde{V}_q(K_g,Q).
\end{equation*}

First, we have the following entropy-type integral estimate.

\begin{lemma}\label{lemEnt1}
Assume $\mu$ is a finite Borel measure on $\uS$ which is not
concentrated in any closed hemisphere of $\uS$. 
Then for any sequence of
positive support functions $\set{h_k}\subset C^+(\uS)$,
there exists a subsequence $\set{h_{k'}}$,
two small positive numbers $\epsilon_0,\delta_0\in(0,1)$,
and an integer $k_0$, such that for any $k'\geq k_0$,
\begin{equation} \label{eq:95}
  \frac{1}{|\mu|} \int_{\uS} \log h_{k'} \dd\mu
  \geq \log\left(
    \frac{\delta_0}{2} 
    (\max h_{k'})^{\epsilon_0}
    \cdot (\min h_{k'})^{1-\epsilon_0}
  \right).
\end{equation} 
\end{lemma}

\begin{proof}{}
For simplicity, write  
\begin{equation*}
  R_k=\max h_k,
  \qquad
  r_k=\min h_k.
\end{equation*}

\textup{(a)}  
We first consider the case $r_k=1$ for every $k$.
For each $h_k$, assume $R_k$ is attained at some $x_k\in\uS$, namely $R_k=h_k(x_k)$.
Without loss of generality, we can assume that
\begin{equation}\label{eq:94}
  \lim_{k\to+\infty} x_k=\tilde{x}\in\uS.
\end{equation}

For each $\delta\in(0,1)$, let
\begin{equation*}
  \Omega_{\delta}=\set{x\in\uS : x\cdot\tilde{x}>\delta},
\end{equation*}
which is increasing when $\delta\searrow0^+$, and
\begin{equation*}
  \lim_{\delta\to0^+} \Omega_{\delta}
  =\set{x\in\uS : x\cdot\tilde{x}>0}.
\end{equation*}
Therefore, we have
\begin{equation*}
  \lim_{\delta\to0^+} \frac{\mu(\Omega_{\delta})}{|\mu|}
  = \frac{\mu\left(\set{x\in\uS : x\cdot\tilde{x}>0}\right)}{|\mu|}
  >0,
\end{equation*}
where the inequality is due to the assumption that 
$\mu$ is not concentrated in any closed hemisphere of $\uS$. 
Thus, there exist two small positive numbers $\epsilon_0,\delta_0\in(0,1)$,
such that
\begin{equation}\label{eq:96}
  \frac{\mu(\Omega_{\delta_0})}{|\mu|}
  >\epsilon_0.
\end{equation}

Recalling \eqref{eq:94}, one can find a large $k_0$, such that
\begin{equation*}
  |x_k-\tilde{x}|
  <\frac{\delta_0}{2},
  \quad \forall\,k\geq k_0.
\end{equation*}
In the following proof, we always assume that $k\geq k_0$.
Noting $R_kx_k\in K_{h_k}$,
by the definition of support function, we have
\begin{equation}\label{eq:97}
  h_k(x) \geq R_kx_k\cdot x,
  \quad x\in\uS.
\end{equation}
Recalling $\tilde{x}\cdot x>\delta_0$ for any $x\in\Omega_{\delta_0}$,
there is
\begin{equation*}
  \begin{split}
    x_k\cdot x
    &\geq \tilde{x}\cdot x-|(x_k-\tilde{x})\cdot x| \\
    &\geq \tilde{x}\cdot x-|x_k-\tilde{x}| \\
    &>\frac{\delta_0}{2},
  \end{split}
\end{equation*}
which together with \eqref{eq:97} yields
\begin{equation}\label{eq:98} 
  h_k(x) > \frac{\delta_0}{2} R_k,
  \quad \forall\,x\in\Omega_{\delta_0}.
\end{equation}

Now recalling that $h_k\geq r_k=1$ on $\uS$, we have
\begin{equation} \label{eq:99}
  \begin{split}
    \frac{1}{|\mu|} \int_{\uS} \log h_{k} \dd\mu 
    &\geq \frac{1}{|\mu|} \int_{\Omega_{\delta_0}} \log h_{k} \dd\mu \\
    &\geq \frac{\mu(\Omega_{\delta_0})}{|\mu|} 
    \log\left( \frac{\delta_0}{2} R_k \right) \\
    &= \log \left((\delta_0/2)^m R_k^m \right),
  \end{split}
\end{equation}
where $m=\frac{\mu(\Omega_{\delta_0})}{|\mu|}$ is written for simplicity.
On account of \eqref{eq:96}, we have $\epsilon_0<m<1$.
By virtue of $\delta_0<1$ and $R_k\geq1$,
the above inequality is reduced to
\begin{equation}\label{eq:101}
  \frac{1}{|\mu|} \int_{\uS} \log h_{k} \dd\mu 
  \geq \log \left(
    \frac{\delta_0}{2} R_k^{\epsilon_0}
  \right),
\end{equation}
which is just our conclusion \eqref{eq:95} for the case $r_k=1$.

\textup{(b)}  
For the general case,
applying the proved estimate \eqref{eq:101} to the new sequence
$\set{h_k/r_k}$, we have
\begin{equation*}
  \frac{1}{|\mu|} \int_{\uS} \log\frac{h_{k}}{r_k} \dd\mu 
  \geq \log \left(
    \frac{\delta_0}{2}  
    \left( R_k/r_k \right)^{\epsilon_0}
  \right), 
\end{equation*}
which is just the general \eqref{eq:95}.
The proof of this lemma is completed.
\end{proof}

Then, we prove a sharp estimate about $\VTq$ when $q<0$.

\begin{lemma}\label{lem318}
Let $Q$ be a star body in $\R^n$.
For any convex body $K\subset\R^n$  
containing the origin in its interior,
we have 
\begin{equation} \label{eq:100}
  \VTq(K,Q)\approx 
  (\min\rho_K)^q
  \qquad \text{for }q<0,
\end{equation}
where the ``$\approx$'' means the ratio of the two sides has positive upper and lower
bounds depending only on $n$, $q$, $\min\rho_Q$ and $\max\rho_Q$.
\end{lemma}

\begin{proof}{}
Assume $r=\min_{\uS}\rho_K$ is attained at some point $\tilde{u}\in\uS$.
Then, we have
\begin{equation*}
  r=\rho_K(\tilde{u})=h_K(\tilde{u}). 
\end{equation*}
By \eqref{eq:102}, there is
\begin{equation*} 
  \frac{1}{\rho_K(u)} 
  \geq \frac{u\cdot \tilde{u}}{h_K(\tilde{u})}
  = \frac{u\cdot \tilde{u}}{r},
  \quad \forall\,u\in\uS.
\end{equation*}
Therefore, we have for $q<0$ that
\begin{equation}\label{eq:103}
  \begin{split}
    \int_{\uS}\rho_K^q(u)\dd u
    &\geq \int_{u\cdot\tilde{u}>0}\rho_K^q(u)\dd u \\
    &\geq \int_{u\cdot\tilde{u}>0}
    \left(\frac{u\cdot \tilde{u}}{r} \right)^{-q} \dd u \\
    &= r^q \int_{\set{u\in\uS : u_1>0}} u_1^{-q} \dd u \\
    &=C_{n,q} \,r^q,
  \end{split}
\end{equation}
where $u_1$ denotes the first coordinate of $u$, and $C_{n,q}$ is a positive
number depending only on $n$ and $q$.
On the other hand, there is obviously that
\begin{equation*}
  \int_{\uS}\rho_K^q(u)\dd u
  \leq \int_{\uS} r^q \dd u
  =\omega_{n-1}\,r^q.
\end{equation*}
Combining this inequality and \eqref{eq:103}, we obtain
\begin{equation} \label{eq:104}
  \int_{\uS}\rho_K^q(u)\dd u
  \approx r^q,
  \quad \text{when }q<0.
\end{equation}
Now, we have
\begin{equation*}
  \begin{split}
    \widetilde{V}_q(K,Q)
    &= \frac{1}{n}\int_{\uS}\rho_K^q(u)\rho_Q^{n-q}(u)\dd u \\
    &\approx \int_{\uS}\rho_K^q(u)\dd u,
  \end{split}
\end{equation*}
which together with \eqref{eq:104} yields the conclusion \eqref{eq:100}.
\end{proof}

Now, we prove the minimizing problem \eqref{minP1} has a solution
under very weak constraints. 

\begin{lemma}\label{lem319}
Assume $q<0$, $Q$ is a star body in $\R^n$,
and $\mu$ is a finite Borel measure on $\uS$ which is not
concentrated in any closed hemisphere of $\uS$.
Then the minimizing problem \eqref{minP1} has a solution $h$.
In addition, the solution $h$ is the support function of $K_h$. 
\end{lemma}

\begin{proof}{}
Similar to the argument in the first paragraph
in the proof of Lemma \ref{lem300}, 
we can assume that a sequence of positive support functions
$\set{h_k}\subset C^+(\uS)$ is a minimizing sequence of \eqref{minP1}.
Due to the zeroth-order homogeneity of $J$,
we assume that $\min_{\uS}h_k=1$ for every $k$.

Applying Lemma \ref{lemEnt1} to the sequence $\set{h_k}$,
there exists a subsequence, still denoted by $\set{h_k}$,
two small positive numbers $\epsilon_0,\delta_0\in(0,1)$,
and an integer $k_0$, such that for any $k\geq k_0$,
\begin{equation*} 
  \frac{1}{|\mu|} \int_{\uS} \log h_{k} \dd\mu
  \geq \log\left(
    \frac{\delta_0}{2} 
    (\max h_{k})^{\epsilon_0}
  \right).
\end{equation*} 
Since $q<0$ and $\min\rho_{K_{h_k}}=\min h_k=1$, by Lemma \ref{lem318},
there is
\begin{equation*} 
  \VTq(K_{h_k},Q)\approx 1,
  \quad \forall\,k.
\end{equation*}
Therefore, we have for $k\geq k_0$ that
\begin{equation}\label{eq:105}
  \begin{split}
    J[h_k]
    &=
    \frac{1}{|\mu|} \int_{\uS} \log h_k \dd\mu
    -\frac{1}{q}\log\widetilde{V}_q(K_{h_k},Q) \\
    &\geq \log\left( \frac{\delta_0}{2} 
      (\max h_{k})^{\epsilon_0} \right)
    -C_0,
  \end{split} 
\end{equation}
where $C_0$ is a positive constant depending only on
$n$, $q$, $\min\rho_Q$ and $\max\rho_Q$.

Recalling that $\set{h_k}$ is a minimizing sequence of \eqref{minP1},
without loss of generality, one can assume 
\begin{equation*}
  J[h_{k}]<1+J[1], 
  \quad \forall\,k\geq k_0.
\end{equation*}
Here $J[1]= -\frac{1}{q}\log\widetilde{V}_{n-q}(Q)$ is a finite number.
Combining it with \eqref{eq:105}, we obtain 
\begin{equation*}
  \max h_k\leq
  \left(\frac{C_1}{\delta_0}
  \right)^{1/\epsilon_0},
  \quad \forall\,k\geq k_0,
\end{equation*}
where $C_1$ is a positive constant depending only on
$n$, $q$, $\min\rho_Q$ and $\max\rho_Q$.
Recall $\min h_k=1$ for every $k$.
We see that $\set{h_k}$ has uniform positive lower and upper bounds.

Applying Blaschke selection theorem to $\set{h_{k}}$,
there is a subsequence, still denoted by $\set{h_{k}}$,
which uniformly converges to some support function $h$ on $\uS$.
Obviously, $h\in C^+(\uS)$.
Correspondingly, $K_{h_{k}}$ converges to $K_h\in \mathcal{K}_o^n$
which is the convex body determined by $h$.
Recalling the definition of $J$, there is $\lim_{k\to+\infty}J[h_{k}]=J[h]$.
Thus,
\begin{equation*}
  J[h]= \inf\set{ J[g] : g\in C^+(\uS)}.
\end{equation*}  
Therefore, $h$ is a solution to the minimizing problem \eqref{minP1}.
The proof of this lemma is now completed.
\end{proof}

For the minimizer $h$ obtained in Lemma \ref{lem319},
by repeating verbatim the proof of Lemma \ref{lem317},
but without requiring evenness,
one can see that 
$\widetilde{C}_q(c K_{h},Q) =\mu$ for some positive number $c$.
This is precisely the sufficiency part of Theorem \ref{thm4}.
The necessity part is obvious.
Thus, we have proved Theorem \ref{thm4} is true.

\section{The case $q=0$}
\label{sec6}

In this section, we prove Theorem \ref{thm3}.
On account of Lemma \ref{lemVari}, the functional of $\widetilde{C}_0$ is
different from that of $\widetilde{C}_q$ for $q\neq0$. 
Therefore, we consider a different minimizing problem for the case $q=0$:
\begin{equation}\label{minP2}
  \inf\set{ \tilde{J}[g] : g\in C_e^+(\uS)},
\end{equation} 
where
\begin{equation} \label{Jt}
  \tilde{J}[g]=
  \frac{1}{|\mu|} \int_{\uS} \log g\, \dd\mu 
  -\frac{1}{\VOL(Q)} \widetilde{E}(K_g,Q).
\end{equation}
Here, $K_g$ is still the Alexandrov body associated with $g$,
and $\widetilde{E}$ is the dual mixed entropy given in \eqref{dmEnt}.
As before, our main work is to prove that
the minimizing problem \eqref{minP2} has a solution.

Fortunately, the part $\frac{1}{|\mu|} \int_{\uS} \log g\, \dd\mu$ here
can be still handled via Lemma \ref{lemEnt1}.
We only need to deal with the term $\widetilde{E}$.

\begin{lemma}\label{lem320}
Let $Q$ be a star body in $\R^n$.
For any origin-symmetric convex body $K\subset\R^n$, we have 
\begin{equation} \label{eq:108}
  \widetilde{E}(K,Q)\leq
  \VOL(Q) \log (\min\rho_K) +C_{Q},
\end{equation}
where $C_Q$ is a positive constant depending only on $n$ and $Q$.
\end{lemma}

\begin{proof}{}
Since $K$ is origin-symmetric, we can assume that
$r=\min_{\uS}\rho_K$ is attained at points $\pm\tilde{u}\in\uS$.
Then, we have
\begin{equation*}
  r=\rho_K(\pm\tilde{u})=h_K(\pm\tilde{u}). 
\end{equation*}
By \eqref{eq:102}, there is
\begin{equation*} 
  \frac{1}{\rho_K(u)} 
  \geq \frac{u\cdot(\pm\tilde{u})}{h_K(\pm\tilde{u})}
  = \frac{\pm u\cdot \tilde{u}}{r},
  \quad \forall\,u\in\uS.
\end{equation*}
Therefore, we obtain
\begin{equation*} 
  \frac{1}{\rho_K(u)} 
  \geq \frac{|u\cdot \tilde{u}|}{r},
  \quad \forall\,u\in\uS,
\end{equation*}
namely,
\begin{equation*} 
  \log \rho_K(u) \leq
  \log r +\log|u\cdot\tilde{u}|^{-1},
  \quad \forall\,u\cdot\tilde{u}\neq0.
\end{equation*}
Thus, 
\begin{equation}\label{eq:109}
  \begin{split}
    \frac{1}{n} \int_{\uS} \rho_Q^n(u) \log \rho_K(u) \dd u
    &\leq
    \frac{1}{n} \int_{\uS} \rho_Q^n(u) \log r \dd u
    +\frac{1}{n} \int_{\uS} \rho_Q^n(u) \log|u\cdot\tilde{u}|^{-1} \dd u \\
    &\leq
    \VOL(Q) \log r 
    +\frac{1}{n} (\max\rho_Q)^n \int_{\uS} \log|u\cdot\tilde{u}|^{-1} \dd u \\
    &=
    \VOL(Q) \log r 
    +\frac{1}{n} (\max\rho_Q)^n \int_{\uS} \log|u_1|^{-1} \dd u \\
    &=
    \VOL(Q) \log r +C_1,
  \end{split}
\end{equation}
where $u_1$ denotes the first coordinate of $u$, and $C_1$ is a positive
number depending only on $n$ and $\max\rho_Q$.
Now recalling the definition of $\widetilde{E}$ in \eqref{dmEnt}, we have
\begin{equation*}
  \widetilde{E}(K,Q)
  = \frac{1}{n}
  \int_{\uS} \rho_Q^n(u) \log \rho_K(u) \dd u
  - \frac{1}{n}
  \int_{\uS} \rho_Q^n(u) \log \rho_Q(u) \dd u,
\end{equation*}
which together with \eqref{eq:109} yields the conclusion \eqref{eq:108}.
\end{proof}

Before prove the existence of a minimizer, we note that
$\tilde{J}$ in \eqref{Jt} is homogeneous of degree zero.
Namely, for $g\in C^+(\uS)$ and $\lambda>0$, there is
\begin{equation} \label{eq:107}
  \tilde{J}[\lambda g]=\tilde{J}[g].
\end{equation}
In fact, by the definition of $\widetilde{E}$ in \eqref{dmEnt}, we have
\begin{equation*}
  \widetilde{E}(K_{\lambda g},Q)
  = \widetilde{E}(K_{g},Q) +\VOL(Q)\log\lambda,
\end{equation*}
which obviously implies \eqref{eq:107}.

\begin{lemma}\label{lem321}
Assume $Q$ is an origin-symmetric star body in $\R^n$,
and $\mu$ is a finite even Borel measure on $\uS$ which 
is not concentrated on any great sub-sphere of $\uS$.
Then, the minimizing problem \eqref{minP2} has a solution $h$.
In addition, the solution $h$ is the support function of $K_h$. 
\end{lemma}

\begin{proof}{}
Assume $\set{g_k}\subset C_e^+(\uS)$ is a minimizing sequence of \eqref{minP2}.
Since $g_k$ is even, the Alexandrov body $K_{g_k}$ is origin-symmetric.
Let $h_k$ be the support function of $K_{g_k}$.
Then we have that $h_k\in C_e^+(\uS)$, $h_k\leq g_k$, and $K_{h_k}=K_{g_k}$.
Therefore,
\begin{equation*}
  \begin{split}
    \tilde{J}[h_k] &=
    \frac{1}{|\mu|} \int_{\uS} \log h_k \dd\mu 
    -\frac{1}{\VOL(Q)} \widetilde{E}(K_{h_k},Q) \\
    &\leq
    \frac{1}{|\mu|} \int_{\uS} \log g_k \dd\mu
    -\frac{1}{\VOL(Q)} \widetilde{E}(K_{g_k},Q) \\
    &= \tilde{J}[g_k],
  \end{split} 
\end{equation*}
implying that
$\set{h_k}$ is also a minimizing sequence of \eqref{minP2}.
Recalling the zeroth-order homogeneity of $\tilde{J}$ given in \eqref{eq:107},
we can assume that $\min_{\uS}h_k=1$ for every $k$.
  
Since $\mu$ is even and
not concentrated on any great sub-sphere of $\uS$,
it is not concentrated in any closed hemisphere of $\uS$. 
Then applying Lemma \ref{lemEnt1} to the sequence $\set{h_k}$,
there exists a subsequence, still denoted by $\set{h_k}$,
two small positive numbers $\epsilon_0,\delta_0\in(0,1)$,
and an integer $k_0$, such that for any $k\geq k_0$,
\begin{equation*} 
  \frac{1}{|\mu|} \int_{\uS} \log h_{k} \dd\mu
  \geq \log\left(
    \frac{\delta_0}{2} 
    (\max h_{k})^{\epsilon_0}
  \right).
\end{equation*} 
Noting that $K_{h_k}$ is origin-symmetric,
and $\min\rho_{K_{h_k}}=\min h_k=1$,
by Lemma \ref{lem320}, there is
\begin{equation*} 
  \widetilde{E}(K_{h_k},Q)\leq C_{Q},
\end{equation*}
where $C_Q$ is a positive constant depending only on $n$ and $Q$.
Therefore, we have for $k\geq k_0$ that
\begin{equation}\label{eq:110}
  \begin{split}
    \tilde{J}[h_k] &=
    \frac{1}{|\mu|} \int_{\uS} \log h_k \dd\mu 
    -\frac{1}{\VOL(Q)} \widetilde{E}(K_{h_k},Q) \\
    &\geq \log\left( \frac{\delta_0}{2} 
      (\max h_{k})^{\epsilon_0} \right)
    -\frac{C_Q}{\VOL(Q)}.
  \end{split} 
\end{equation}

Recalling that $\set{h_k}$ is a minimizing sequence of \eqref{minP2},
without loss of generality, one can assume 
\begin{equation}\label{eq:111}
  \tilde{J}[h_{k}]<1+\tilde{J}[1], 
  \quad \forall\,k\geq k_0.
\end{equation}
Note that
\begin{equation*}
  \tilde{J}[1]
  = \frac{1}{n\VOL(Q)}
  \int_{\uS} \rho_Q^n(u) \log \rho_Q(u) \dd u
\end{equation*}
is a finite number depending only on $n$ and $Q$.
Combining \eqref{eq:110} and \eqref{eq:111}, we obtain 
\begin{equation*}
  \max h_k\leq
  \left(\frac{C_1}{\delta_0}
  \right)^{1/\epsilon_0},
  \quad \forall\,k\geq k_0,
\end{equation*}
where $C_1$ is a positive constant depending only on $n$ and $Q$.
Recall $\min h_k=1$ for every $k$.
We see that $\set{h_k}$ has uniform positive lower and upper bounds.

Applying Blaschke selection theorem to $\set{h_{k}}$,
there is a subsequence, still denoted by $\set{h_{k}}$,
which uniformly converges to some support function $h$ on $\uS$.
Obviously, $h\in C_e^+(\uS)$.
Correspondingly, $K_{h_{k}}$ converges to $K_h\in \mathcal{K}_e^n$
which is the convex body determined by $h$.
Recalling the definition of $\tilde{J}$, there is
$\lim_{k\to+\infty}\tilde{J}[h_{k}]=\tilde{J}[h]$.
Thus,
\begin{equation*}
  \tilde{J}[h]= \inf\set{ \tilde{J}[g] : g\in C_e^+(\uS)}.
\end{equation*}  
Therefore, $h$ is a solution to the minimizing problem \eqref{minP2}.
The proof of this lemma is now completed.
\end{proof}

Now, we can prove the sufficiency part of Theorem \ref{thm3}.

\begin{lemma}\label{lem324}
Assume $Q$ is an origin-symmetric star body in $\R^n$.
If $\mu$ is a finite even Borel measure on $\uS$ which
is not concentrated on any great sub-sphere of $\uS$
and $|\mu|=\VOL(Q)$,
then there exists an origin-symmetric convex body $K$ in $\R^n$
such that
\begin{equation*}
  \widetilde{C}_0(K,Q,\cdot)=\mu.
\end{equation*}
\end{lemma}

\begin{proof}
Applying Lemma \ref{lem321},   
the minimizing problem \eqref{minP2} has a solution $h\in C_e^+(\uS)$,
which is the support function of $K_h$.
For any given continuous even function $\varphi\in C(\uS)$, let
\begin{equation*}
  g_t=h+t\varphi \quad \text{for small } t\in\R.
\end{equation*}
Since $h\in C^+_{e}(\uS)$, for $t$ sufficiently small $g_t\in C^+_{e}(\uS)$ as well.
By Lemma \ref{lemVari}, we have
\begin{equation} \label{eq:112}
  \frac{\dd}{\dd t} \widetilde{E}(K_{g_t},Q) \Big|_{t=0}
  = \int_{\uS} \varphi h^{-1} \dd\widetilde{C}_0(K_{h},Q).
\end{equation}

Write $\tilde{J}(t)=\tilde{J}[g_t]$.
Then $\tilde{J}(0)=\tilde{J}[h]$.
Since $h$ is a minimizer of \eqref{minP2}, there is
\begin{equation*}
  \tilde{J}(t)\geq \tilde{J}(0)
  \quad \text{for any small } t\in\R,
\end{equation*}
which together with \eqref{eq:112} and
the definition of $\tilde{J}$ in \eqref{Jt}
yields that
\begin{equation*}
  \begin{split}
    0&= \frac{\dd}{\dd t} \tilde{J}(t) \Big|_{t=0} \\
    &= \frac{\dd}{\dd t} \left(
      \frac{1}{|\mu|} \int_{\uS} \log g_t \dd\mu 
      -\frac{1}{\VOL(Q)} \widetilde{E}(K_{g_t},Q)
    \right) \bigg|_{t=0} \\
    &=
    \frac{1}{|\mu|} \int_{\uS} \varphi h^{-1} \dd\mu 
    -\frac{1}{\VOL(Q)}
    \int_{\uS} \varphi h^{-1} \dd\widetilde{C}_0(K_{h},Q).
  \end{split}
\end{equation*}
Note that this equality holds for arbitrary even function $\varphi$,
and that $\mu$, $\widetilde{C}_0(K_h,Q)$ are even Borel measures.
Therefore, we obtain
\begin{equation*}
  \frac{1}{|\mu|} \mu 
  =\frac{1}{\VOL(Q)}
  \widetilde{C}_0(K_{h},Q).
\end{equation*}
By the assumption that $|\mu|=\VOL(Q)$, we have
\begin{equation*}
  \mu = \widetilde{C}_0(K_{h},Q),
\end{equation*}
which is just the conclusion of this lemma.
\end{proof}

The necessity part of Theorem \ref{thm3} is obvious.
The proof of Theorem \ref{thm3} is completed.


\bibliographystyle{siam}

\end{document}